\documentclass[12pt]{amsart}
\usepackage[pdftex]{graphicx}
\usepackage{amsmath,amssymb,stmaryrd}
\usepackage{amsthm, amsfonts}
\usepackage{enumerate}
\usepackage{comment}
\usepackage[all]{xy}
\usepackage[pdftex]{color}
\usepackage{hyperref}
\usepackage[margin=27truemm]{geometry}
\usepackage{stmaryrd}
\SetSymbolFont{stmry}{bold}{U}{stmry}{m}{n}

\theoremstyle{plain}
\newtheorem{thm}{Theorem}[section]
\theoremstyle{definition}
\newtheorem{dfn}[thm]{Definition}
\newtheorem{rem}[thm]{Remark}
\theoremstyle{plain}
\newtheorem{lem}[thm]{Lemma}
\newtheorem{prop}[thm]{Proposition}

\theoremstyle{definition}

\newcommand{\R}{\mathbb{R}}

\newcommand{\Z}{\mathbb{Z}}
\newcommand{\K}{\mathbb{K}}
\newcommand{\G}{\mathbb{G}}

\numberwithin{equation}{section}

\title{On the Upsilon invariant in grid homology}

\author{Hajime Kubota}
\date{}

\begin{document}

\begin{abstract}
The Upsilon invariant is a concordance invariant in knot Floer homology.
F\"{o}ldv\'{a}ri \cite{The-knot-invariant-Upsilon-using-grid-homologies} reconstructed the Upsilon invariant using grid homology.
We prove that the Upsilon invariant in knot Floer homology and one in grid homology are equivalent.
Furthermore, we show some properties of the Upsilon invariant in the framework of grid homology.
\end{abstract}

\maketitle

\section{Introduction}

Ozsv\'{a}th-Stipsicz-Szab\'{o} \cite{Concordance-homomorphisms-from-knot-Floer-homology} defined the Upsilon invariant $\Upsilon_K\colon[0,2]\to\mathbb{R}$ for a knot $K\subset S^3$.
In this paper, we call it \textit{the original Upsilon invariant}.
The original Upsilon invariant is a smooth concordance invariant and gives bounds for knot invariants such as the slice genus, concordance genus, and the unknotting number.
This invariant was first defined by constructing a family of chain complexes $\{tCFK(K)\}_{t\in[0,2]}$ from the knot Floer complex $CFK^-(K)$ and taking their homologies.
Later, Livingston \cite{Notes-on-the-knot-concordance-invariant-upsilon} gave an alternative formulation of the original Upsilon invariant using the full knot Floer complex $CFK^\infty(K)$.

Grid homology is a combinatorial reconstruction of knot Floer homology developed by Manolescu-Ozsv\'{a}th-Szab\'{o}-Thurston \cite{oncombinatorial}.
F\"{o}ldv\'{a}ri \cite{The-knot-invariant-Upsilon-using-grid-homologies} reconstructed the Upsilon invariant using grid homology and showed that it is an invariant for topological knots.
In this paper, we call it \textit{the grid Upsilon invariant}.
Later, the author \cite{Concordance-invariant-Upsilon-for-balanced-spatial-graphs-using-grid-homology} proved the concordance invariance of the grid Upsilon invariant in the framework of grid homology.

The definition of the grid Upsilon invariant and the definition of the original Upsilon invariant are very similar but there are subtle differences.
Therefore it is not obvious whether they are the same, even if there is a chain homotopy equivalence between grid complexes and knot Floer complexes.
F\"{o}ldv\'{a}ri \cite{The-knot-invariant-Upsilon-using-grid-homologies} raised the question of whether the grid Upsilon invariant is equivalent to the original Upsilon invariant and whether the known properties of the original Upsilon invariant are proved in the framework of grid homology.

Even though the equivalence of the grid and the original upsilon invariant was unknown, grid homology is useful to study the original Upsilon invariant.
Ozsv\'{a}th-Stipsicz-Szab\'{o} \cite{Unoriented-knot-Floer-homology-and-the-unoriented-four-ball-genus} gave a combinatorial description of the upsilon invariant $\upsilon(K):=\Upsilon_K(1)$.
They studied the behavior of the modified grid chain complex corresponding to the complex $tCFK(K)$ of $t=1$ under surface cobordisms and showed that $\upsilon(K)$ gives a lower bound for the smooth four-dimensional crosscap number of $K$.
Sano-Sato \cite{An-algorithm-for-computing-the-Υ-invariant-and-the-d-invariants-of-Dehn-surgeries} computed the original Upsilon invariant for prime knots with up to $11$ crossings.
We remark that they used grid representations of knots but a different approach from this paper.

In this paper, we show that the grid Upsilon invariant equals the original Upsilon invariant using the equivalence of knot Floer homology and grid homology.
In addition, we prove some known properties of the original Upsilon invariant in the framework of grid homology.
Because the proof is done algebraically, we will not review the precise definition of grid homology.
See \cite{oncombinatorial} or \cite{Grid-homology-for-knots-and-links} for details.

\begin{thm}
\label{thm:Upsilon}
For a knot $K \subset S^3$, the grid Upsilon invariant coincides with the original Upsilon invariant.
\end{thm}

We briefly explain the idea of the proof.
The original Upsilon invariant is defined by the following two steps \cite{Concordance-homomorphisms-from-knot-Floer-homology}:
First, we apply the algebraic operation called $t$-modification to obtain a one-parameter family of chain complexes $\{tCFK(K)\}_{t\in[0.2]}$ from the knot Floer complex $CFK^-(K)$.
Then the Upsilon invariant $\Upsilon_K$ at $t$ is defined by the maximal grading of the homogeneous, non-torsion element of $H(tCFK(K))$.
The grid Upsilon invariant is defined similarly.
The grid chain complex $GC^-(\G)$ is a finitely generated complex over $\mathbb{F}[U_1,\dots,U_n]$.
The $t$-modified grid chain complex $tGC^-(\G)$ is defined by the resulting complex by applying $t$-modification to $\frac{GC^-(\G)}{U_1=\dots=U_n}$.
Then the grid Upsilon invariant is defined by the maximal grading of the homogeneous, non-torsion element of $H(tGC^-(\G))$.
The following schematic picture describes the situation of the two Upsilon invariants.
\[
\xymatrix@C=40pt{
CFK^-(K)\ar[dd]_{t\mathrm{-modification}}\ar@{-}[r]^{\simeq} & GC^-(\G)\ar[d]^{\mathrm{quotient}} \\
{} & \frac{GC^-(\G)}{U_1=\dots=U_n}\ar[d]^{t\mathrm{-modification}} \\
tCFK(K) \ar[d]_{H(\bullet)} \ar@{.}[r]^? & tGC^-(\G)\ar[d]^{H(\bullet)} \\
\Upsilon_K(t) & \Upsilon_\G(t)
}
\]
The essential parts of the proof are to give a relation between $tCFK(K)$ and $tGC^-(\G)$ using the chain homotopy equivalence $CFK^-(K)\simeq GC^-(\G)$ which implies the equivalence of knot Floer homology and grid homology and to observe that the effect of taking quotient $GC^-(\G)\to\frac{GC^-(\G)}{U_1=\dots=U_n}$ can be ignored in some sense.

If a grid diagram $\G$ represents a knot $K$, its horizontal reflection $\G^*$ represents the mirror image of $K$.
Then there is an isomorphism between $\frac{GC^-(\G^*)}{U_1=\dots=U_n}$ and the dual complex $(\frac{GC^-(\G)}{U_1=\dots=U_n})^*$ with some grading shift.
By lifting it to a relation between $t$-modified complexes $tGC^-(\G^*)$ and $(tGC^-(\G))^*$, we can obtain the following.
\begin{prop}
\label{prop:mirror}
For a knot $K$ and $t\in[0,2]$, the grid Upsilon invariant satisfies
\[
\Upsilon_{m(K)}(t)=-\Upsilon_K(t),
\]
where $m(K)$ is the mirror of $K$.
\end{prop}
This proposition in knot Floer homology can be quickly proved because the taking reflection of the Heegaard diagram induces the chain homotopy equivalent $CFK^-(m(K))\simeq CFK^-(K)^*$ and then $tCFK(m(K))\simeq tCFK(K)^*$.
On the other hand, in grid homology, the situation is a little more complicated since we need to take quotient $GC^-(\G)\to\frac{GC^-(\G)}{U_1=\dots=U_n}$.

It is known that the original Upsilon invariant of an alternating knot is completely determined by the Alexander polynomial and the signature \cite[Theorem 1.14]{Concordance-homomorphisms-from-knot-Floer-homology}.
This is deduced by that knot Floer homology of an alternating knot is also determined by the Alexander polynomial and its signature, and Proposition \ref{prop:mirror} in knot Floer homology.
In fact, grid homology for an alternating knot is also proved to be determined by the Alexander polynomial and the signature (Theorem \ref{thm:quasi-GH-determined}).
Using this and Proposition \ref{prop:mirror}, we obtain that the Upsilon invariant of an alternating knot is completely determined by the Alexander polynomial and the signature in the framework of grid homology.

\begin{thm}
\label{thm:quasi-alt}
If $K$ is an alternating knot (or more generally, a quasi-alternating knot), then
\[
\Upsilon_K(t)=(1-|t-1|)\frac{\sigma(K)}{2},
\]
where $\sigma(K)$ is the signature of $K$.
\end{thm}

\section{Preliminaries}
\subsection{Algebraic terminology}
We fix a base ring $\K$.
(In this paper, we only consider $\K=\Z/2\Z$.)
\begin{dfn}
A \textbf{$\Z$-graded, $\Z$-filtered chain complex over $\K$} is a $\K$-module $C$ with the following structure:
\begin{itemize}
    \item a $\K$-module map $\partial\colon C\to C$ satisfying $\partial\circ\partial=0$;
    \item a $\Z$-grading $C=\bigoplus_{d\in\Z}C_d$ that is compatible with the differential in the sense that $\partial(C_d)\subseteq C_{d-1}$;
    \item the $\Z$ filtration $\dots\subseteq\mathcal{F}_sC\subseteq \mathcal{F}_{s+1}C\subseteq \cdots$ such that $\bigcup_{s\in\Z}\mathcal{F}_sC=C$;
    \item the filtration is compatible with the $\Z$-grading, let $\mathcal{F}_sC_d=(\mathcal{F}_sC)\cap C_d$ ,then $\mathcal{F}_sC=\bigoplus_{d\in\Z}\mathcal{F}_sC_d$;
    \item the filtration is compatible with the differential, $\partial(\mathcal{F}_sC)\subseteq\mathcal{F}_sC$;
    \item the filtration is bounded below.
\end{itemize}

Fix some integer $n\geq 0$.

A \textbf{$\Z$-graded, $\Z$-filtered chain complex over $\K[U_1,\dots,U_n]$} is a $\Z$-graded, $\Z$-filtered chain complex $C$ with $\K$-module maps $U_i\colon C\to C$ for $i=1,\dots, n$ satisfying the followings:
\begin{itemize}
    \item for all $1\leq i,j\leq n$, $U_1$ and $U_j$ commute each other;
    \item each $U_i$ satisfies $\partial\circ U_i=U_i\circ\partial$, $U_i(C_d)\subseteq C_{d-2}$, and $U_i(\mathcal{F}_sC)\subseteq \mathcal{F}_{s-1}C$.
\end{itemize}
\end{dfn}

\begin{dfn}
For a $\Z$-graded, $\Z$-filtered chain complex $(C,\partial)$, the \textbf{associated graded object} is the chain complex
\[
\mathrm{gr}(C)=\bigoplus_{d,s\in\Z}(\mathcal{F}_sC_d/\mathcal{F}_{s-1}C_d),
\]
equipped with the bigrading $\mathrm{gr}(C)_{d,s}=\mathcal{F}_sC_d/\mathcal{F}_{s-1}C_d$ and the differential $\mathrm{gr}(\partial)\colon \mathrm{gr}(C)\to \mathrm{gr}(C)$ induced by $\partial$.
\end{dfn}

\begin{dfn}
Let $f\colon C\to C'$ be a graded, filtered chain map between two $\Z$-graded, $\Z$-filtered chain complexes over $\K[V_1,\dots,V_n]$.
\begin{itemize}
    \item $f$ is called a \textbf{filtered quasi-isomorphism} if it is a graded, filtered chain map whose associated graded map $\mathrm{gr}(f)$ induces an isomorphism on homology.
    \item $f$ is called a \textbf{filtered chain homotopy equivalence} if it is graded, filtered chain map and chain homotopy equivalence.
\end{itemize}
\end{dfn}

\begin{prop}[{\cite[Proposition A.8.1]{Grid-homology-for-knots-and-links}}]
\label{prop:quasi-iso-homotopy}
Let $C$ and $C'$ be two $\mathbb{Z}$-graded, $\mathbb{Z}$-filtered chain complexes over $\mathbb{K}[U]$.
They are filtered chain homotopy equivalent over $\mathbb{K}[U]$ if and only if they are filtered quasi-isomorphic over $\mathbb{K}[U]$.
\end{prop}
Note that the proof of this proposition in \cite{Grid-homology-for-knots-and-links} says that if $f\colon C \to C'$ is a filtered quasi-isomorphism, then $f$ is a filtered chain homotopy equivalence.

\subsection{The t-modified chain complex}
This subsection describes $t$-modification, an algebraic operation to construct a one-parameter family of chain complexes from a $\Z$-graded, $\Z$-filtered chain complex over $\mathbb{F}[U]$.

A set $A\subset\mathbb{R}$ is \textit{well-ordered} if any subset $A'\subset A$ has a minimal element.
Let $\mathbb{R}_{\geq0}$ denote the set of nonnegative real numbers.
\begin{dfn}[{\cite[Definition 3.1]{Concordance-homomorphisms-from-knot-Floer-homology}}]
\label{basedring}
The ring of long power series $\mathcal{R}$ is defined as follows.
$\mathcal{R}$ is the group of formal sums
\[
\left\{\sum_{\alpha\in A}v^\alpha|A\subset\mathbb{R}_{\geq0},A\ \text{is well-ordered}\right\},
\]
where the sum in $\mathcal{R}$ is given by the formula
\[
\left(\sum_{\alpha\in A}v^\alpha\right)+\left(\sum_{\beta\in B}v^\beta\right)=\sum_{\gamma\in C=A\cup B\setminus A\cap B}v^\gamma.
\]
The product is defined by
\[
\left(\sum_{\alpha\in A}v^\alpha\right)\cdot\left(\sum_{\beta\in B}v^\beta\right)=\sum_{\gamma\in A+B}\#\{(\alpha,\beta)\in A\times B|\alpha+\beta=\gamma\}\cdot v^\gamma,
\]
where
\[
A+B=\left\{\gamma|\gamma=\alpha+\beta\ \text{for some }\alpha\in A\text{ and }\beta\in B\right\}.
\]
\end{dfn}

Suppose $C$ is a finitely generated, $\Z$-graded, $\Z$-filtered chain complex over $\mathbb{F}[U]$.
Let $\mathbf{x}$ be a generator of $C$ over $\mathbb{F}[U]$, with the $\Z$-grading $M(\mathbf{x})$.
Since multiplication by $U$ drops the $\Z$-grading by 2, any element of the $\Z$-grading $M(\mathbf{x})-1$ is written as a linear sum of elements of the form $U^{\frac{M(\mathbf{y})-M(\mathbf{x})+1}{2}}\mathbf{y}$, where $\mathbf{y}$ is a generator.
Therefore the differential on $C$ can be written as
\begin{equation}
\label{dd}
\partial(\mathbf{x})=\sum_{\mathbf{y}}c_{\mathbf{x},\mathbf{y}}\cdot U^{\frac{M(\mathbf{y})-M(\mathbf{x})+1}{2}}\mathbf{y},
\end{equation}
where $c_{\mathbf{x},\mathbf{y}}\in\{0,1\}$.

\begin{dfn}[{\cite[Definition 4.1]{Concordance-homomorphisms-from-knot-Floer-homology}}]
\label{dfn:t-mod}
Suppose that $C$ is a finitely generated, $\Z$-graded, $\Z$-filtered chain complex over $\mathbb{F}[U]$.
We regard that $\mathcal{R}$ contains $\mathbb{F}[U]$ by $U=v^2$.
For $t\in[0,2]$, \textit{The $t$-modified chain complex} $C^t$ of $C$ is defined as follows$\colon$
\begin{itemize}
\item As an $\mathcal{R}$-module, $C^t=C\otimes_{\mathbb{F}[U]}\mathcal{R}$
\item For each generator $\mathbf{x}$ of $C$, define $\mathrm{gr}_t(v^\alpha\mathbf{x})=M(\mathbf{x})-tA(\mathbf{x})-\alpha$
\item Endow the graded module $C^t$ with a differential
\begin{equation}
\label{differential_formal}
\partial_t(\mathbf{x})=\sum_{\mathbf{y}}c_{\mathbf{x},\mathbf{y}}\cdot v^{\mathrm{gr}_t(\mathbf{y})-\mathrm{gr}_t(\mathbf{x})+1}\mathbf{y},
\end{equation}
where $c_{\mathbf{x},\mathbf{y}}$ are determined by (\ref{dd}).
\end{itemize}
\end{dfn}

\begin{prop}[{\cite[Proposition 4.4]{Concordance-homomorphisms-from-knot-Floer-homology}}]
\label{prop:f-f^t}
Let $f\colon C\to D$ be a $\Z$-graded, $\Z$-filtered chain map between two chain complexes over $\mathbb{F}[U]$.
If $f$ is a filtered chain homotopy equivalence, then there is a corresponding chain homotopy equivalence $f^t\colon C^t\to D^t$.
\end{prop}

Let $C$ be a finitely generated chain complex over $\mathcal{R}$.
The \textit{dual complex} $C^*=\mathrm{Mor}_{\mathcal{R}}(C,\mathcal{R})$ of $C$ is defined as follows.
As a module, $C^*$ is the module of maps $\phi\colon C\to\mathcal{R}$ which commute with the $\mathcal{R}$-action, in other words, $\phi(r\cdot x)=r\cdot \phi(x)$ for $x\in C$ and $r\in\mathcal{R}$.
The differential $d$ of $C^*$ is determined by
\[
(d\phi)(x)=\phi(\partial x).
\]

Let $\mathrm{gr}(v^\alpha)=-\alpha$ and the degree of a morphism in $C^*$ be $m$ if it sends elements of $C$ of degree $n$ to those of degree $m+n$.
Then if $C$ is a graded chain complex over $\mathcal{R}$, we get the graded dual complex $C^*$.
This construction gives the usual cochain complex with $-1$ times its usual grading.

\begin{prop}[{\cite[Proposition 4.5]{Concordance-homomorphisms-from-knot-Floer-homology}}]
Let $C$ be a finitely generated, $\Z$-graded, $\Z$-filtered chain complex over $\mathbb{F}[U]$.
Then we have
\[
(C^*)^t\cong (C^t)^*,
\]
where $C^*=\mathrm{Hom}(C,\mathbb{F}[U])$.
\end{prop}

\subsection{The Upsilon invariant in knot Floer homology}
For a knot $K\subset S^3$, the knot Floer complex $CFK^-(K)$ is a finitely generated, $\Z$-graded, $\Z$-filtered chain complex over $\mathbb{F}[U]$.
See \cite{Holomorphic-disks-and-knot-invariants} for details for knot Floer complexes.

\begin{dfn}[{\cite[Definition 3.4]{Concordance-homomorphisms-from-knot-Floer-homology}}]
For $t\in[0,2]$, \textit{the $t$-modified knot Floer complex} $tCFK(K)$ is the $t$-modified chain complex $(CFK^-(K))^t$.
\end{dfn}
By definition, the $t$-modified knot Floer complex $tCFK(K)$ is a finitely generated $\mathrm{gr}_t$-graded module over $\mathcal{R}$.

\begin{thm}[{\cite[Theorem 3.5]{Concordance-homomorphisms-from-knot-Floer-homology}}]
As an isomorphism class of $\mathrm{gr}_t$-graded module over $\mathcal{R}$, the homology of $tCFK(K)$ is a knot invariant.
\end{thm}
The proof is quickly given by combining Proposition \ref{prop:f-f^t} and the fact that the filtered chain homotopy type of $CFK^-$ is a knot invariant.

\begin{dfn}[{\cite[Definition 3.6]{Concordance-homomorphisms-from-knot-Floer-homology}}]
For a knot $K\subset S^3$ and $t\in[0,2]$, the \textit{Upsilon invariant} $\Upsilon_K(t)\colon[0,2]\to\R$ is defined by
\[
\Upsilon_K(t):=\mathrm{max}\{\mathrm{gr}_t(x)|x\in H(tCFK(K)),x\textrm{ is homogeneous, non-torsion}\},
\]
where "non-torsion" means that it is non-torsion as an element of $\mathcal{R}$-module.
\end{dfn}

We write some properties of the Upsilon invariant \cite{Concordance-homomorphisms-from-knot-Floer-homology}.

\begin{enumerate}[(1)]
    \item $\Upsilon_K(t)=\Upsilon_K(2-t)$,
    \item $\Upsilon_K(0)=\Upsilon_K(2)=0$,
    \item $\Upsilon_{K_1\#K_2}(t)=\Upsilon_{K_1}(t)+\Upsilon_{K_2}(t)$,
    \item $\Upsilon_{m(K)}(t)=-\Upsilon_K(t)$, where $m(K)$ is the mirror of $K$,
    \item $\Upsilon_{K_+}(t)\leq \Upsilon_{K_-}(t)\leq \Upsilon_{K_+}(t)+t$ for $0\leq t \leq 1$, where $K_+$ and $K_-$ are two knots which differ in a crossing change,
    \item $|\Upsilon_K(t)|\leq t\cdot g_s(K)$, where $g_s(K)$ is the slice genus of $K$.
    \item For a sufficient small $t>0$, the slope of $\Upsilon_K(t)$ equals $-\tau(K)$.
\end{enumerate}
The properties (3) and (6) imply that the Upsilon invariant is a smooth knot concordance invariant.

\subsection{Grid homology for knots}
In this paper, we only briefly describe grid chain complexes because it is sufficient to prove our main theorem.

A grid diagram $\G$ is an $n\times n$ grid of squares some of which are decorated with an $X$- or $O$-marking such that each row and column contains exactly one $X$- and one $O$-marking.
For an $n\times n$ grid diagram $\G$, the \textit{filtered grid chain complex} $\mathcal{GC}^-(\G)$ is the finitely generated, $\Z$-graded, $\Z$-filtered chain complex over $\mathbb{F}[U_1,\dots,U_n]$, where $\mathbb{F}=\mathbb{Z}/2\mathbb{Z}$.
The minus version of \textit{grid homology} of $\G$ is the homology of the associated graded object of $\mathcal{GC}^-(\G)$, regarded as an $\mathbb{F}[U]$-module where the action of $U$ is induced by multiplication by $U_1$.
Let $\widehat{\mathcal{GC}}(\G)=\mathcal{GC}^-(\G)/U_1$ be the quotient complex.
The hat version of grid homology $\widehat{GH}(\G)$ is the homology of the associated graded object of $\widehat{\mathcal{GC}}(\G)$, regarded as an $\mathbb{F}$-vector space.

\begin{thm}[{\cite[Theorem 3.3]{A-combinatorial-description-of-knot-Floer-homology}}]
\label{thm:grid=HFK}
If $\G$ is a grid diagram representing a knot $K$, then the filtered chain homotopy type of $\mathcal{GC}^-(\G)$ and $CFK^-(K)$ are the same, where $\mathcal{GC}^-(\G)$ is regarded as a $\mathbb{Z}$-filtered complex over $\mathbb{F}[U]$ and the action of $U$ is defined to be multiplication by $U_1$.
\end{thm}

\begin{thm}[{\cite[Theorem 1.6]{Grid-diagrams-and-Manolescu's-unoriented-skein-exact-triangle-for-knot-Floer-homology},\cite[Theorem 10.3.1 and Corollary 10.3.2]{Grid-homology-for-knots-and-links}}]
\label{thm:quasi-GH-determined}
If $K$ is an alternating knot, then its grid homologies $\widehat{GH}(K)$ and $GH^-(K)$ are completely determined by the Alexander polynomial $\Delta_K(t)$ and the signature $\sigma(K)$.
In particular, $\widehat{GH}_d(K,s)\neq 0$ if $d-s=\frac{\sigma(K)}{2}$. 
\end{thm}
We remark that Wong \cite{Grid-diagrams-and-Manolescu's-unoriented-skein-exact-triangle-for-knot-Floer-homology} gives a combinatorial proof of the unoriented skein exact triangle and showed that $\widehat{GH}(K)$ is homologically thin.
Then the above theorem can be proved algebraically \cite[Theorem 10.3.1 and Corollary 10.3.2]{Grid-homology-for-knots-and-links}.

\subsection{The Upsilon invariant in grid homology}
\label{sec:grid-homology-for-knots}
This section provides the grid Upsilon invariant.
The grid Upsilon invariant is first defined by F\"{o}ldv\'{a}ri \cite{The-knot-invariant-Upsilon-using-grid-homologies}, and reformulated by the author \cite{Concordance-invariant-Upsilon-for-balanced-spatial-graphs-using-grid-homology}.

The \textit{$t$-modified grid complex} $tGC^-(\G)$ and the grid Upsilon invariant were first defined by F\"{o}ldv\'{a}ri \cite{The-knot-invariant-Upsilon-using-grid-homologies} for $t\in[0,2]\cap\mathbb{Q}$, by giving them directly and concretely, in other words, without $\mathcal{GC}^-(\G)$.
Later, the author reformulated them for the whole $t\in[0,2]$ as follows.

\begin{dfn}
\label{dfn:t-grid}
For an $n\times n$ grid diagram $\G$, the quotient complex $\frac{\mathcal{GC}^-(\G)}{U_1=\dots=U_n}$ is a finitely generated $\mathbb{Z}$-graded, $\mathbb{Z}$-filtered chain complex over $\mathbb{F}[U]$.
For $t\in[0,2]$, the \textit{t-modified grid complex} of $\G$ is defined by $tGC^-(\G)=\left(\frac{\mathcal{GC}^-(g)}{U_1=\dots=U_n}\right)^t$ (Definition \ref{dfn:t-mod}).
\end{dfn}

\begin{prop}[{\cite[Proposition 4.3]{Concordance-invariant-Upsilon-for-balanced-spatial-graphs-using-grid-homology}}]
\label{U=t}
If $t\in[0,2]\cap\mathbb{Q}$, the complex $\left(\frac{\mathcal{GC}^-(g)}{U_1=\dots=U_n}\right)^t$ of Definition \ref{dfn:t-grid} is isomorphic to F\"{o}ldv\'{a}ri's t-modified grid complex $tGC^-(\G)$ as graded chain complexes.
\end{prop}

\begin{dfn}
Let $\G$ be a grid diagram representing a knot $K\subset S^3$.
The \textbf{grid Upsilon invariant} $\Upsilon_\G(t)\colon[0,2]\to\R$ is defined by for $t\in[0,1]$,
\[
\Upsilon_\G(t):=\mathrm{max}\{\mathrm{gr}_t(x)|x\in H(tGC^{-}(\G)),x\textrm{ is homogeneous, non-torsion}\},
\]
where “non-torsion" means that it is non-torsion as an element of $\mathcal{R}$-module.
For $t\in(1,2]$, let $\Upsilon_\G(t)=\Upsilon_\G(2-t)$.
\end{dfn}

For a technical reason, the gird Upsilon invariant at $t\in(1,2]$ is not defined directly.
The reason is related to the fact that the homology of $tGC^-(\G)$ depends on the size $n$ of the grid diagram $\G$.

\begin{thm}[{\cite[Corollary 5.3]{The-knot-invariant-Upsilon-using-grid-homologies}}]
Let $\G$ and $\G'$ be two grid diagrams representing the same knot.
Then we have $\Upsilon_\G(t)=\Upsilon_{\G'}(t)$ for $t\in[0,2]$, in other words, the grid Upsilon invariant is a knot invariant.
\end{thm}

\begin{thm}[{\cite[Theorem 1.3]{Concordance-invariant-Upsilon-for-balanced-spatial-graphs-using-grid-homology}}]
The grid Upsilon invariant is a smooth concordance invariant of knots.
\end{thm}

\begin{rem}
Ozsv\'{a}th-Stipsicz-Szab\'{o} \cite{Unoriented-knot-Floer-homology-and-the-unoriented-four-ball-genus} first use grid homology to study the upsilon invariant $\upsilon_K=\Upsilon_K(1)$.
They observed the complex $GC'(\G)$ that is obtained from $\mathcal{GC}^-(\G)$ by collapsing the grading $\delta=M-A$, where $M$ is the $\Z$-grading and $A$ is the $\Z$-filtration level.
The $t$-modified complex $tGC^-(\G)$ is the extension of $GC'(\G)$ because $GC'(\G)$ corresponds $tGC^-(\G)$ of $t=1$.
\end{rem}

\section{The proof of the main theorems}

Let $\mathcal{W}$ be the two-dimensional vector space with one generator in grading zero and filtration level zero, and another in grading $-1$ and filtration level $-1$.
For a $\mathbb{Z}$-graded, $\mathbb{Z}$-filtered chain complex $\mathcal{C}$ over $\mathbb{F}[U]$, we get a new chain complex $\mathcal{C}\otimes \mathcal{W}$, with
\[
\mathcal{F}_i(\mathcal{C}\otimes \mathcal{W})_d=\mathcal{F}_i\mathcal{C}_d\oplus \mathcal{F}_{i+1}\mathcal{C}_{d+1}.
\]

\begin{lem}
\label{lem:u1=un}
Let $\G$ be an $n \times n$ grid diagram.
Then $\mathcal{GC}^-(\G)\otimes \mathcal{W}^{\otimes (n-1)}$ and $\frac{\mathcal{GC}^-(\G)}{U_1=\dots=U_n}$ are filtered chain homotopy equivalent.
\end{lem}
\begin{proof}
Combine \cite[Lemma 14.1.11]{Grid-homology-for-knots-and-links} and \cite[Proposition A.8.1]{Grid-homology-for-knots-and-links}.
\end{proof}

\begin{proof}[Proof of Theorem \ref{thm:Upsilon}]
Let $\G$ be an $n \times n$ grid diagram for $K$.
It is sufficient to consider the case where $t\in[0,1]$ since the grid Upsilon invariant satisfies $\Upsilon_\G(t)=\Upsilon_\G(2-t)$ by definition.

By Theorem \ref{thm:grid=HFK}, the filtered grid chain complex $\mathcal{GC}^-(\G)$ (regarded as $\mathbb{F}[U]$-module by letting $U=U_1$) and $CFK^-(K)$ are filtered chain homotopy equivalent as complexes over $\mathbb{F}[U]$.
Using Theorem \ref{thm:grid=HFK} and Lemma \ref{lem:u1=un}, we have
\[
\frac{\mathcal{GC}^-(\G)}{U_1=\dots=U_n} \simeq \mathcal{GC}^-(\G)\otimes \mathcal{W}^{\otimes (n-1)} \simeq CFK^-(K)\otimes \mathcal{W}^{\otimes (n-1)}.
\]
Then Proposition \ref{prop:f-f^t} gives
\[
tGC^-(\G)=\left(\frac{\mathcal{GC}^-(\G)}{U_1=\dots=U_n}\right)^t \simeq (CFK^-(K)\otimes \mathcal{W}^{\otimes (n-1)})^t \simeq tCFK(K)\otimes W_t^{\otimes (n-1)},
\]
where $W_t$ is the two-dimensional graded vector space $W_t\cong\mathbb{F}_{0}\oplus\mathbb{F}_{-1+t}$, where their indices describe t-gradings.
Therefore we obtain $\Upsilon_{\G}(t)=\Upsilon_K(t)$ since the grading shift from $W_t$ does not affect the value of $\Upsilon_{\G}(t)$ for $t\in[0,1]$.
\end{proof}

\begin{rem}
The chain homotopy equivalent $tGC^-(\G) \simeq tCFK(K)\otimes W_t^{\otimes (n-1)}$ in the proof is the generalization of \cite[Corollary 2.8]{Unoriented-knot-Floer-homology-and-the-unoriented-four-ball-genus}
\end{rem}

\begin{lem}[{\cite[Lemma 14.2.1]{Grid-homology-for-knots-and-links}}]
\label{lem:finite-complex}
The $\Z$-graded, $\Z$-filtered chain complex $\mathcal{GC}^-(\G)$ is filtered chain homotopy equivalent to the complex which is a free module over $\mathbb{F}[U]$ with rank $\dim_\mathbb{F}\widehat{GH}(\G)$. 
\end{lem}
We remark that $\dim_\mathbb{F}\widehat{GH}(\G)$ is finite and that this lemma is given algebraically without holomorphic theory.

\begin{proof}[Proof of Proposition \ref{prop:mirror}]
Let $\G$ be a grid diagram representing $K$ and $\G^*$ be the horizontal reflection of $\G$.
Then $\G^*$ represents $m(K)$.
By the same argument as the proof of \cite[Propositions 7.4.3]{Grid-homology-for-knots-and-links}, reflection $\G^*\mapsto\G$ induces an isomorphism of $\Z$-graded, $\Z$-filtered chain complexes
\[
\frac{\mathcal{GC}^-(\G^*)}{U_1=\dots=U_n}\to \left( \frac{\mathcal{GC}^-(\G)}{U_1=\dots=U_n} \right)^* \llbracket n-1,n-1\rrbracket,
\]
where $C^*=\mathrm{Hom}(C,\mathbb{F}[U])$ is the dual complex of $C$ and $\llbracket a,b\rrbracket$ denotes a shift of the grading and the filtration.

Let $\mathcal{C}$ be a finitely generated $\Z$-graded, $\Z$-filtered chain complex over $\mathbb{F}[U]$.
For $t\in[0,1]$, let 
\[
\Upsilon_\mathcal{C}(t):=\mathrm{max}\{\mathrm{gr}_t(x)|x\in H(\mathcal{C}^t),x\textrm{ is homogeneous, non-torsion}\},
\]
and $\Upsilon_\mathcal{C}(t)=\Upsilon_\mathcal{C}(2-t)$ for $t\in(1,2]$.
Clearly $\Upsilon_{\mathcal{C}}\equiv\Upsilon_\mathcal{D}$ if $\mathcal{C}\simeq\mathcal{D}$.
Furthermore, we have $\Upsilon_{\mathcal{C}\otimes\mathcal{W}}\equiv\Upsilon_\mathcal{C}$ since the grading shift from $\mathcal{W}$ does not affect the value of $\Upsilon$ for $t\in[0,1]$.

Because the operation $t$-modification is defined for finitely generated complexes but $\mathcal{GC}^-(\G)$ is not finitely generated as $\mathbb{F}[U_i]$-module for $i=1,\dots,n$, we need to use Lemma \ref{lem:finite-complex}.
Let $\mathcal{C}$ be a finitely generated chain complex which is filtered chain homotopy equivalent to $\mathcal{GC}^-(\G)$ obtained by applying Lemma \ref{lem:finite-complex}.
Let $\mathcal{W}^*$ be the two-dimensional vector space with one generator in grading zero and filtration level zero, and another in grading $1$ and filtration level $1$.
By Lemma \ref{lem:u1=un}, we have
\begin{align*}
\frac{\mathcal{GC}^-(\G^*)}{U_1=\dots=U_n} &\cong \left( \frac{\mathcal{GC}^-(\G)}{U_1=\dots=U_n} \right)^* \llbracket n-1,n-1\rrbracket\\
 &\simeq (\mathcal{GC}^-(\G)\otimes \mathcal{W}^{\otimes (n-1)})^*\llbracket n-1,n-1\rrbracket\\
 &\simeq (\mathcal{GC}^-(\G))^* \otimes (\mathcal{W}^*)^{\otimes (n-1)}\llbracket n-1,n-1\rrbracket\\
 &\simeq (\mathcal{GC}^-(\G))^* \otimes \mathcal{W}^{\otimes (n-1)}\\
 &\simeq \mathcal{C}^* \otimes \mathcal{W}^{\otimes (n-1)}.
\end{align*}

By the universal coefficient theorem and the construction of dual complexes of $t$-modified complexes, we have $\Upsilon_{\mathcal{C}^*}\equiv-\Upsilon_\mathcal{C}$.
Then we have 
\[\Upsilon_{\G^*}\equiv\Upsilon_{\mathcal{C}^*\otimes \mathcal{W}^{\otimes (n-1)}}\equiv\Upsilon_{\mathcal{C}^*}\equiv-\Upsilon_\mathcal{C}\equiv-\Upsilon_\G,
\]
and thus $\Upsilon_{m(K)}\equiv-\Upsilon_K$.
\end{proof}

\begin{proof}[Proof of Theorem \ref{thm:quasi-alt}]
The idea is the same as the proof of \cite[Theorem 1.14]{Concordance-homomorphisms-from-knot-Floer-homology}.
Since we can use Theorem \ref{thm:quasi-GH-determined} and Proposition \ref{prop:mirror} that are combinatorially proved, the same argument as the proof of \cite[Theorem 1.14]{Concordance-homomorphisms-from-knot-Floer-homology} works.

Let $\G$ be an $n\times n$ grid diagram for $K$ and $\mathcal{C}$ be a finitely generated filtered complex over $\mathbb{F}[U]$ obtained from $\mathcal{GC}^-(\G)$ by Lemma \ref{lem:finite-complex}.
Since $\mathcal{C}\simeq tCFK(K)\otimes \mathcal{W}^{\otimes (n-1)}$, we have
\[
H(\mathcal{C})/\mathrm{Tors}(H(\mathcal{C}))\cong  \mathbb{F}[U]\otimes \mathcal{W}^{\otimes (n-1)} \cong \bigoplus_{i=0}^{n-1} (\mathbb{F}[U]\llbracket i,i\rrbracket) ^{c_i},
\]
where $c_i=\binom{n-1}{i}$.

Now, with a slight modification, we apply the same argument as the proof of \cite[Theorem 1.14]{Concordance-homomorphisms-from-knot-Floer-homology}.
By Proposition \ref{prop:mirror}, we can assume that $-\frac{\sigma(K)}{2}\geq 0$.
We can take $2^{n-1}$ sequences of elements $x^i_0,\dots,x^i_n,y^i_0,\dots,y^i_{n-1}$ and integers $k_i$ for $i=1,\dots,2^{n-1}$ in $\mathcal{C}$ satisfying the followings:
\begin{itemize}
    \item For each $i$ and $j$, we have $\partial y^i_j=U x^i_j+x^i_{j+1}$,
    \item For each $i$ and $j$, we have $A(x^i_j)=-\frac{\sigma(K)}{2}-2j-k_i$ and $M(x^i_j)=-2j-k_i$.
    \item Let $d_i=c_0+\dots+c_{i}$. Then we have $k_1=0$, $k_2=\dots=k_{d_1}=1$, $k_{d_1+1}=\dots=k_{d_2}=2$, \dots, and $k_{d_{n-2}+1}=\dots=k_{d_n-1}=n-1$.
\end{itemize}
Then, the elements $x^1_0, \dots, x^{2^{n-1}}_0$ represent non-torsion generators for each summand of $H(\mathcal{C})/\mathrm{Tors}(H(\mathcal{C}))\cong \mathbb{F}[U]^{2^{n-1}}$.
Furthermore, any non-torsion class must contain at least one of the $x^i_j$.
When $t\in[0,1]$, we have 
\[
\max_{i}\mathrm{gr}_t(x^i_0)=\mathrm{gr}_t(x^0_0)=\frac{\sigma(K)}{2}\cdot t.
\]
Therefore we have
\[
\Upsilon_\mathcal{C}(t)=\mathrm{gr}_t(x^0_0)=\frac{\sigma(K)}{2}\cdot t.
\]
The symmetry $\Upsilon(2-t)=\Upsilon(t)$ completes the proof.
\end{proof}

\section{Acknowledgement}
I am grateful to my supervisor, Tetsuya Ito, for helpful discussions and corrections.
This work was supported by JST, the establishment of university fellowships towards
the creation of science technology innovation, Grant Number JPMJFS2123.

\bibliography{grid}
\bibliographystyle{amsplain} 

\end{document}